\documentclass{svproc}
\usepackage{url}
\usepackage{amsmath, amsfonts, amssymb, graphicx, color}
\usepackage{mathtools} 

\begin{document}
\mainmatter              
\title{On Bivariate Jacobi Konhauser Polynomials}
\titlerunning{Bivariate Jacobi Konhauser Polynomials}  
%
\author{İlkay Onbaşı Elidemir\inst{1} \and Mehmet Ali Özarslan\inst{1} }
\authorrunning{İ.O. Elidemir and M.A. Özarslan} 
%
\tocauthor{İlkay Onbaşı Elidemir, and  Mehmet Ali Özarslan}
\institute{Eastern Mediterranean University, Department of Mathematics, Gazimagusa, TRNC, Mersin 10, Turkey,\\
\email{ilkay.onbasi@emu.edu.tr}
\\
\email{mehmetali.ozarslan@emu.edu.tr}}

\maketitle              

\begin{abstract}
Recently, Özarslan and Elidemir (2023) proposed a novel methodology for constructing two-variable biorthogonal polynomial families by utilizing one-variable biorthogonal and orthogonal polynomial families. The primary objective of the paper is to introduce novel class of two-variable biorthogonal polynomials namely bivariate Jacobi Konhauser polynomials. We analyze various essential properties of these polynomials, such as their biorthogonality, operational formula, generating function, and integral representation. In addition, we examine the transformations of these polynomials under the Laplace transform, as well as their behavior under fractional integral and derivative operators. In the context of these polynomials, we introduce a new class of bivariate Jacobi-Konhauser-Mittag-Leffler (JKML) functions and derive corresponding properties for them. Furthermore, we define an integral operator whose kernel incorporates the bivariate JKML function and compute the images of power and exponential functions under its action. Finally, we explore the transformation of this operator under the Laplace transform.
\end{abstract}

\section{Introduction}

Orthogonal polynomials represent a class of polynomials that are mutually orthogonal with respect to a given inner product, typically defined over a specific interval or domain. Orthogonal polynomials are a rich and important area of study in mathematics, providing powerful tools for approximating functions, solving differential equations, and analyzing data. Their well-defined properties and relationships with weight functions make them essential in both theoretical and applied contexts. 
Orthogonal polynomials are defined w.r.t (with respect to) a weight function $w(t)$ on an interval $[a,b]$ such that
\begin{equation*}
\int_a^bw(t)p_n(t)p_m(t)dt=0 \quad  m\neq n,
\end{equation*}
where the degrees of polynomials $p_n(t)$ and $p_m(t)$ are  $n$ and $m$ respectively.

Laguerre polynomials, denoted as $L_n(t)$, are defined on the interval $[0,\infty)$ and are orthogonal w.r.t the weight function $w(t)=e^{-t}$. Generalized Laguerre polynomials extend the concept of classical Laguerre polynomials to accommodate a broader range of parameters. They are often denoted as $L_n^\alpha(x)$ and are defined for a non-negative integer $n$ and a parameter $\alpha \leq 0$. They have the orthogonality relation as follow
\begin{equation}
\int_{0}^{\infty} e^{-t}t^{\alpha}L_n^\alpha(t)L_m^\alpha(t)dt=\frac{\Gamma(n+\alpha+1)}{n!}\delta_{nm} \label{lag}.
\end{equation}

Hermite polynomials are a fundamental class of orthogonal polynomials with significant applications across mathematics and physics. They have the relation as
\begin{equation}
\int_{-\infty}^{\infty}e^{-t^2}  H_n(t)H_m(t)dt= 2^{n}n!\sqrt{\pi}\delta_{nm}. \label{her}
\end{equation}

Jacobi polynomials are a class of classical orthogonal polynomials that are defined on the interval $[-1,1]$ w.r.t a specific weight function. They have the following orthogonality relation
\begin{equation}
\begin{aligned}
{}& \int_{-1}^1 (1-t)^{\alpha}(1+t)^{\beta}P_n^{(\alpha,\beta)}(t)P_m^{(\alpha,\beta)}(t)dt \\&
=\begin{cases}
      0 ,& \text{if} \quad  n\neq m\\
      \frac{2^{1+\alpha+\beta}\Gamma(1+ \alpha+n)\Gamma(1+\beta+n)}{(2n+\alpha+\beta+1)\Gamma(1+\alpha+\beta+n)n!} & \text{if}  \quad  n=m \label{jacc}
    \end{cases}.
\end{aligned}
\end{equation}

Biorthogonal polynomials are pairs of polynomials that satisfy a certain orthogonality condition w.r.t different weight functions. Biorthogonal polynomials find applications in numerous fields of mathematics and physics, especially in orthogonal polynomial theory and numerical analysis. They often arise in problems involving differential equations, approximation theory, and spectral methods. Examples of biorthogonal polynomial pairs include Legendre and Chebyshev polynomials, as well as other families depending on the specific weight function involved. (see \cite{carl}, \cite{kon}, \cite{jde}, \cite{hc}).
\\
Konhauser polynomials are a specific class of biorthogonal polynomials that arise in various mathematical contexts, particularly in the study of orthogonal polynomial theory. They are named after the mathematician V. Konhauser, who contributed to their formulation and study.

\begin{equation}
Z_N^{\beta}(t;\Lambda)=\frac{\Gamma(1+\beta+\Lambda N)}{N!}\sum_{r=0}^N\frac{(-N)_rt^{\Lambda r}}{r!\Gamma(\beta+1+\Lambda r)}, \label{11}
\end{equation}
\begin{equation}
Y_N^{(\beta)}(t;\Lambda)=\frac{1}{N!}\sum_{k=0}^N\frac{t^k}{k!}\sum_{j=0}^k (-1)^j\binom{k}{j}\bigg(\frac{j+\beta+1}{\Lambda}\bigg)_N. \label{2}
\end{equation}
The polynomials \eqref{11} and\eqref{2} constitute a pair of biorthogonal polynomials derived from the  polynomials $L_n^\alpha(x)$.
They have the following biorthogonality relation 
\begin{equation}
\int_{0}^{\infty} e^{-t}t^{\beta}Z_n^\beta(t;\Lambda)Y_m^\beta(t;\Lambda)dx=\frac{\Gamma(\Lambda n+\beta+1)}{n!}\delta_{nm}. \label{konn}
\end{equation} 

Bivariate orthogonal polynomials are a generalization of univariate orthogonal polynomials to two variables. They are polynomials in two variables $t$ and $z$ that satisfy orthogonality conditions w.r.t a given weight function defined over a region in the two-dimensional plane. Bivariate orthogonal polynomials extend the concept of orthogonality to functions of two variables, allowing for rich structures and applications in various mathematical fields. They maintain the properties of orthogonal polynomials while adapting to the complexity introduced by an additional variable.
\\
Bivariate biorthogonal polynomials generalize the concepts of biorthogonality and orthogonality to functions of two variables. Their structure allows for a wide range of applications in mathematics, physics, and engineering, making them a significant area of study in polynomial theory.
The general method for constructing bivariate biorthogonal polynomials was introduced in \cite{ilk}.

\begin{remark} \label{remark}
Let $w_1(t)$ denote a weight function over the interval $(a,b)$, and let $R_n(t)$ and $S_n(t)$ represent biorthogonal polynomials associated with the  polynomials $r(t)$ and $s(t)$, respectively. In other terms, these polynomials fulfill the biorthogonality relation.
\begin{equation*}
\int_a^b w_1(t)R_m(t)S_n(t)dt=J_{n,m}:=\begin{cases}
      0 ,& \text{if} \quad  n\neq m\\
      J_{n,n} , & \text{if}  \quad  n=m
    \end{cases}.
\end{equation*}
Also, 
\begin{equation*}
K_n(t)=\sum_{i=0}^n D_{n,i}[r(t)]^i
\end{equation*}
with
\begin{equation*}
\int_c^d w_2(t)K_m(t)K_n(t)dt=||K_n||^2\delta_{nm}.
\end{equation*}
The bivariate polynomials 
\begin{equation}
P_n(t,z)=\sum_{s=0}^n \frac{D_{n,s}}{J_{n-s,n-s}}[r(t)]^s R_{n-s}(z) \quad \text{and} \quad Q_n(t,z)=K_n(t)\sum_{j=0}^{max\{m,n\}}  S_j(z)\label{1}
\end{equation}
are biorthogonal w.r.t the weight function $w_1(t)w_2(z)$ over the domain $(a,b)\times(c,d)$.
\end{remark}

 Two dimensional Laguerre-Konhauser polynomials introduced  in \cite{1} by Bin-Saad as follows
\begin{equation*}
\prescript{}{\kappa}L_n^{(\alpha,\beta)}(x,y)=n!\sum_{s=0}^n \sum_{r=0}^{n-s}\frac{(-1)^{s+r}x^{r+\alpha}y^{\kappa s+\beta}}{s!r!(n-s-r)!\Gamma(\alpha+r+1)\Gamma(\kappa s+\beta+1)}   
\end{equation*}
$(\alpha,\beta, \kappa=1,2,...).$ \\
The polynomials $\prescript{}{\kappa}L_n^{(\alpha,\beta)}(x,y)$  expressed in terms of  $L_n^\alpha(x)$  and $Z_n^\beta(y;\kappa)$l respectively,
\begin{equation*}
_{\kappa}L_n^{(\alpha,\beta)}(x,y)=n!\sum_{s=0}^n\frac{(-1)^s x^{s+\alpha}y^\beta Z_{n-s}^{\beta}(y;\kappa)}{s!\Gamma(\alpha+s+1)\Gamma(\kappa n-\kappa s+\beta+1)}
\end{equation*}
and
\begin{equation*}
_{\kappa}L_n^{(\alpha,\beta)}(x,y)=n!\sum_{s=0}^n\frac{(-1)^s x^{\alpha}y^{\kappa s+\beta}L_{n-s}^\alpha(x)}{s!\Gamma(\alpha+n-s+1)\Gamma(\kappa s+\beta+1)}.
\end{equation*}
\par

In \cite{cemo}, Özarslan and Kürt defined a  polynomials set as
\begin{equation*}
_{\kappa}\mathfrak{L}_n^{(\alpha,\beta)}(x,y)=L_n^\alpha(x)\sum_{s=0}^{max\{m,n\}}  Y_s^{(\beta)}(y;\kappa).
\end{equation*}

From \eqref{lag} and \eqref{konn} it can be readily observed that the two sets of polynomials $_{\kappa}\mathfrak{L}n^{(\alpha,\beta)}(t,z)$ and ${\kappa}L_n^{(\alpha,\beta)}(t,z)$ are biorthonormal. Indeed,
\begin{equation*}
\int_{0}^{\infty}\int_{0}^{\infty}e^{-t-z}{_{\kappa}\mathfrak{L}_m^{(\alpha,\beta)}(t,z)}_{\kappa}L_n^{(\alpha,\beta)}(t,z)dtdz=\delta_{nm},
\end{equation*}
where $\delta_{nm}$ is the Kronecker's delta.
\par
Motivated by the above result, in \cite{ilk} the authors introduced new classes of  bivariate birothogonal polynomial families, where the first family named as, Hermite-Konhauser polynomials.
\begin{equation*}
_{\upsilon}H_n^{\varrho}(x,y)=\sum_{s=0}^{[\frac{n}{2}]}\sum_{r=0}^{n-s}\frac{(-1)^s(-n)_{2s}(-n)_{s+r}}{(-n)_s\Gamma(\varrho+1+\upsilon r)s!r!}(2x)^{n-2s}y^{\upsilon r} \label{3}
\end{equation*}
where  $\varrho>-1$ and $\upsilon=1,2\dots$. \\
And the second class of polynomial families defined as,
\begin{equation*}
Q_m(x,y)=H_m(x)\sum_{j=0}^{max\{m,n\}}  Y_j^{(\varrho)}(y;k). \label{6}  
\end{equation*}
where $H_m(x)$ is the Hermite polynomials.
It can be easily observed that  $_{\upsilon}H_n^{\varrho}(t,z)$ and $Q_m(t,z)$ are biorthogonal. Indeed
\begin{equation*}
\int_0^{\infty }\int_{-\infty}^{\infty}e^{-t^2-z}  y^\varrho_{\phantom{1}\upsilon}H_n^{\varrho}(t,z)Q_m(t,z)dtdz=\begin{cases}
      0 ,& \text{if} \quad  n\neq m\\
      2^{n}n!\sqrt{\pi} & \text{if}  \quad  n=m
    \end{cases}.  \label{5}
\end{equation*}

Corresponding to  the polynomials $_{\Upsilon}L_n^{(\varkappa,\varrho)}(x,y)$ a bivariate Mittag  Leffler (ML) functions $E^{(\gamma)}_{\varkappa,\varrho,\upsilon}(x,y)$ introduced  by  Özarslan and Kürt \cite{cemo} as
\begin{equation*}
E_{\varkappa,\varrho,\upsilon}^{(\gamma)}(x,y)=\sum_{r=0}^\infty\sum_{s=0}^\infty \frac{(\gamma)_{r+s}x^ry^{\Upsilon s}}{r!s!\Gamma(\varkappa+r)\Gamma(\varrho+\upsilon s)},    \label{cemo}
\end{equation*}
$(\varkappa, \varrho, \gamma \in\mathbb{C}, \mathcal{R}(\varkappa), \mathcal{R}(\varrho), \mathcal{R}(\upsilon)>0), $ where
\begin{equation*}
_{\phantom{1}\Upsilon}L_n^{(\varkappa,\varrho)}(x,y)=x^{\varkappa} y^{\varrho} E_{\varkappa+1,\varrho+1,\Upsilon}^{(-n)}(x,y).
\end{equation*}
And corresponding to the polynomials $_{\upsilon}H_n^{\varrho}(x,y)$, Özarslan and Elidemir \cite{ilk} defined a bivariate ML functions $E_{\varrho,\upsilon}^{(\gamma_1;\gamma_2;\gamma_3)}(x,y)$ by 
\begin{equation*}
E_{\varrho,\upsilon}^{(\gamma_1;\gamma_2;\gamma_3)}(x,y)=\sum_{s=0}^\infty\sum_{r=0}^\infty\frac{(\gamma_1)_{2s}(\gamma_2)_{s+r}x^sy^{\upsilon r}}{(\gamma_3)_s\Gamma(\varrho+\upsilon r)r!s!},    \label{9}
\end{equation*}
$(\upsilon, \varrho, \gamma_1,\gamma_2, \gamma_3 \in\mathbb{C}, \mathcal{R}(\gamma_1)>0,  \mathcal{R}(\gamma_2)>0,  \mathcal{R}(\gamma_3)>0,  \mathcal{R}(\varrho)>0,\mathcal{R}(\upsilon)>0),$ where
\begin{equation*}
{}_{\phantom{1}\upsilon}H_n^{\varrho}(x,y)=(2x)^nE_{\varrho+1,\upsilon}^{(-n;-n;-n)}\bigg(\frac{-1}{4x^2},y\bigg). \label{10}
\end{equation*}

The Mittag-Leffler (ML) function is a special function that generalizes the exponential function and has applications in various areas of mathematics, physics, and engineering, particularly in the study of fractional calculus and complex analysis. Gösta Mittag-Leffler introduced the classical ML function. Since then, numerous generalizations of the function $E_{\varkappa}(x)$ have been developed (see \cite{book}, \cite{luchko}, \cite{luchko2}), including versions that involve two or more variables rather than a single variable.The  ML function has become a significant class of special functions, particularly due to its connection with fractional calculus. Building on the contributions of Prabhakar \cite{prab}, these functions began to appear in the kernels of various fractional calculus operators (see \cite{kilbas}). Recently, bivariate versions of the ML functions have gained notable applications in the applied sciences \cite{natur}. Consequently, there has been an increasing interest in investigating various forms of bivariate ML functions, as well as fractional calculus operators that integrate these functions into their kernels, particularly in the last few years. (see \cite{ng}, \cite{garg}, \cite{kilbas2}, \cite{kilbas3}, \cite{shukla}, \cite{aa}, \cite{rk}, \cite{ma}, \cite{ma1}, \cite{ck}, \cite{kc}, \cite{arren}, \cite{new}, \cite{new2}, \cite{new3}).

The organization of this study is as follows: In Section 2, we define a new bivariate polynomial known as the bivariate Jacobi-Konhauser polynomial using Remark \ref{remark}, and we introduce the corresponding bivariate Jacobi-Konhauser Mittag-Leffler (JKML) function. In Section 3, the Riemann Liouville (R-L) double fractional integral and derivative of both the JKML function and the bivariate Jacobi-Konhauser polynomials are computed. Section 4 outlines the derivation of the double Laplace transforms and offers operational representations for both the bivariate JKML function and the bivariate Jacobi-Konhauser polynomials. In Section 5, we derive the generating functions for the bivariate Jacobi-Konhauser polynomials and derive integral representations for both the bivariate Jacobi-Konhauser polynomials and the bivariate JKML function. Finally, in Section 6, we introduce an integral operator that contains the bivariate JKML function in its kernel.

\section{Bivariate Jacobi Konhauser Polynomials}

In this section, firstly  we remember some properties of Jacobi polynomials and then we define the 2D- Jacobi Konhauser polynomials.
\par
There  are several hypergeometric representations for the one variable Jacobi polynomial, 

\begin{align}
\mathcal{P}_n^{(\alpha,\beta)}(x)&=\frac{(1+\alpha)_n}{n!}\prescript{}{2}F_1\left[-n,1+\alpha+\beta+n,1+\alpha,\frac{1-x}{2}\right], \label{Jacobi1} \\
\mathcal{P}_n^{(\alpha,\beta)}(x)&=\frac{(1+\alpha)_n}{n!}\left(\frac{x+1}{2}\right)^n\prescript{}{2}F_1\left[-n,-\beta-n,1+\alpha,\frac{x-1}{x+1}\right], \label{Jacobi2} \\
\mathcal{P}_n^{(\alpha,\beta)}(x)&=\frac{(-1)^n(1+\beta)_n}{n!}\prescript{}{2}F_1\left[-n,1+\alpha+\beta+n,1+\beta,\frac{1+x}{2}\right] , \label{Jacobi3} \\
\mathcal{P}_n^{(\alpha,\beta)}(x)&=\frac{(1+\beta)_n}{n!}\left(\frac{x-1}{2}\right)^n\prescript{}{2}F_1\left[-n,-\alpha-n,1+\beta,\frac{x+1}{x-1}\right], \label{Jacobi4}
\end{align}
where $\alpha>-1$ , $\beta>-1$ and $ \prescript{}{2}F_1\left[a,b,c,x\right]=\sum_{s=0}^\infty \frac{(a)_s(b)_s x^s}{s!(c)_s}.$
\par

In the following definition, we specify the explicit form of the bivariate Jacobi-Konhauser polynomials with the aid of Remark 1. To obtain the explicit form we choose  $r(x)=\frac{1-x}{2}$ , $R_n(x)=Z_n^{\beta}(x;\kappa) $, $S_n(x)=Y_n^{(\beta)}(x;\kappa)$ and $K_n(x)=\mathcal{P}_n^{(\alpha,\beta)}(x)$ with the hypergeometric  representation \eqref{Jacobi1}.
\begin{definition}
The bivariate Jacobi-Konhauser polynomials are described as
\begin{equation}
\mathcal{\mathcal{\prescript{}{\kappa}P}}_n^{(\alpha,\beta)}(x,y)=\frac{\Gamma(1+\alpha+n)}{n!}\sum_{s=0}^n\sum_{r=0}^{n-s}\frac{(-n)_{s+r}(1+\alpha+\beta+n)_s}{s!r!\Gamma(1+\alpha+s)\Gamma(\beta+1+\kappa r)}\left(\frac{1-x}{2}\right)^sy^{\kappa r} \label{jac}
\end{equation}
where $\alpha>-1$ , $\beta>-1$ and $\kappa=1,2\dots$.
\end{definition}

\begin{remark}
For the  polynomials $\mathcal{\prescript{}{\kappa}P}_n^{(\alpha,\beta)}(x,y)$ the following representations also hold:
\begin{itemize}
\item
choosing $R(x)=\frac{x-1}{x+1}$ and using \eqref{Jacobi2} gives
\begin{equation}
\begin{aligned}
\mathcal{\prescript{}{\kappa}P}_n^{(\alpha,\beta)}(x,y){}& =\frac{\Gamma(1+\alpha+n)}{n!}\left(\frac{x+1}{2}\right)^n\sum_{s=0}^n\sum_{r=0}^{n-s}\frac{(-n)_{s+r}(-\beta-n)_s}{s!r!\Gamma(1+\alpha+s)\Gamma(\beta+1+\kappa r)}\\&
\times \left(\frac{x-1}{x+1}\right)^sy^{\kappa r}; \label{jac2}
\end{aligned}
\end{equation}
\item
 choosing $R(x)=\frac{1+x}{2}$ and using \eqref{Jacobi3} gives
\begin{equation}
\begin{aligned}
\mathcal{\prescript{}{\kappa}P}_n^{(\alpha,\beta)}(x,y){}& =\frac{(-1)^n\Gamma(1+\beta+n)}{n!}\sum_{s=0}^n\sum_{r=0}^{n-s}\frac{(-n)_{s+r}(1+\alpha+\beta+n)_s}{s!r!\Gamma(1+\beta+s)\Gamma(\beta+1+\kappa r)} \\&
\times \left(\frac{1+x}{2}\right)^sy^{\kappa r}; \label{jac3}
\end{aligned}
\end{equation}
\item
choosing $R(x)=\frac{x+1}{x-1}$ and using \eqref{Jacobi4} gives
\begin{equation}
\begin{aligned}
\mathcal{\prescript{}{\kappa}P}_n^{(\alpha,\beta)}(x,y){}&=\frac{\Gamma(1+\beta+n)}{n!}\left(\frac{x-1}{2}\right)^n\sum_{s=0}^n\sum_{r=0}^{n-s}\frac{(-n)_{s+r}(-\alpha-n)_s}{s!r!\Gamma(1+\beta+s)\Gamma(\beta+1+\kappa r)}\\& \times
\left(\frac{x+1}{x-1}\right)^sy^{\kappa r} \label{jac4};
\end{aligned}
\end{equation}
\end{itemize}
\end{remark}

\begin{theorem}
The polynomials, $\mathcal{\prescript{}{\kappa}P}_n^{(\alpha,\beta)}(x,y)$, can be expressed as
\begin{equation}
\begin{aligned}
\mathcal{\prescript{}{\kappa}P}_n^{(\alpha,\beta)}(x,y){}&=\Gamma(1+\alpha+n)\sum_{s=0}^n\frac{(-1)^s(1+\alpha+\beta+n)_s}{s!\Gamma(1+\alpha+s)\Gamma(1+\beta+\kappa n-\kappa s)}\\&
\times \bigg(\frac{1-x}{2}\bigg)^s Z_{n-s}^\beta(y;\kappa). \label{kon}
\end{aligned}
\end{equation}
\end{theorem}
\begin{proof}
The proof is based on the explicit expressions of the polynomials  $\mathcal{\prescript{}{\kappa}P}_n^{(\alpha,\beta)}(x,y)$ and $Z_{n}^\beta(x;\kappa)$.
\end{proof}
\begin{theorem}
The following biorthogonality conditons holds true for bivariate Jacobi Konhauser polynomials 
\begin{equation*}
\begin{aligned}
{}& \int_0^{\infty }\int_{-1}^1(1-x)^\alpha(1+x)^\beta e^{-y} y^\beta\mathcal{\prescript{}{\kappa}P}_n^{(\alpha,\beta)}(x,y)Q_m(x,y)dxdy
\\& =\begin{cases}
      0 ,& \text{if} \quad  n\neq m\\
      \frac{2^{1+\alpha+\beta}\Gamma(1+ \alpha+n)\Gamma(1+\beta+n)}{(2n+\alpha+\beta+1)\Gamma(1+\alpha+\beta+n)n!} & \text{if}  \quad  n=m
    \end{cases},
\end{aligned}
\end{equation*}
where
\begin{equation}
Q_m(x,y)=P_m^{(\alpha,\beta)}(x)\sum_{j=0}^{max\{m,n\}}  Y_j^{(\beta)}(y;k). \label{q}
\end{equation}
\end{theorem}
\begin{proof}
Without loss of generality let $m>n$. By substituting $\mathcal{\prescript{}{\kappa}P}_n^{(\alpha,\beta)}(x,y)$ from \eqref{kon} and $Q_m(x,y)$  from \eqref{q} into the left-hand side, we obtain
\begin{equation*}
\begin{aligned}
{}&\int_0^{\infty }\int_{-1}^1  (1-x)^\alpha(1+x)^\beta e^{-y} y^\beta \Gamma(1+\alpha+n)\\ &
\times \sum_{s=0}^n\frac{(-1)^s(1+\alpha+\beta+n)_s}{s!\Gamma(1+\alpha+s)\Gamma(1+\beta+\kappa n-\kappa s)}\bigg(\frac{1-x}{2}\bigg)^s Z_{n-s}^\beta(y;\kappa)  P_m^{(\alpha,\beta)}(x,y)\\&
\times \sum_{j=0}^m Y_j^{(\beta)}(y;k)dxdy \\&
=\int_{-1}^1  (1-x)^\alpha(1+x)^\beta  \Gamma(1+\alpha+n)\sum_{s=0}^n\frac{(-1)^s(1+\alpha+\beta+n)_s}{s!\Gamma(1+\alpha+s)(n-s)!}\\& 
\times \bigg(\frac{1-x}{2}\bigg)^sP_m^{(\alpha,\beta)}(x,y)dx \\&
=\int_{-1}^1  (1-x)^\alpha(1+x)^\beta P_n^{(\alpha,\beta)}(x,y)P_m^{(\alpha,\beta)}(x,y)dx.
\end{aligned} 
\end{equation*}
Whence we obtain the result from the orthogonality relation of the Jacobi polynomials.
\end{proof}
From the perspective of the following  representation,
\begin{equation*}
F_{k,l,m}^{n,p,q}\begin{bmatrix}( a_n):(b_p);(c_q);&\\ & x;y \\ (d_k) :(f_l);(g_m); & \end{bmatrix}=\sum_{r,s=0}^\infty \frac{\prod_{i=1}^n (a_i)_{r+s}\prod_{i=1}^p(b_i)_{r}\prod_{i=1}^q (c_i)_{s}x^sy^r}{\prod_{i=1}^k (d_i)_{r+s}\prod_{i=1}^l (f_i)_{r}\prod_{i=1}^m (g_i)_{s}s!r!},
\end{equation*}
which is Kampe de Feriet's double hypergeometric series \cite{kampee}, we derive the following theorem for the polynomials $\mathcal{\prescript{}{\kappa}P}_n^{(\alpha,\beta)}(x,y)$.

\begin{theorem}
The  polynomials $\mathcal{\prescript{}{\kappa}P}_n^{(\alpha,\beta)}(x,y)$ have the following Kampe de Feriet's double hypergeometric representation  
\begin{equation}
\mathcal{\prescript{}{\kappa}P}_n^{(\alpha,\beta)}(x,y)=\frac{(1+\alpha)_n}{\Gamma(1+\beta)n!}F_{0,1,\kappa}^{1,1,0}\begin{bmatrix} -n:(1+\alpha+\beta+n);-;&\\ & \frac{1-x}{2};\bigg(\frac{y}{\kappa}\bigg)^\kappa \\ - :(1+\alpha);\bigtriangleup(\kappa;\beta+1); & \end{bmatrix},
\end{equation}
where $\bigtriangleup(k;\sigma)$ denotes the $k$ prameters $\frac{\sigma}{k},\frac{\sigma+1}{k},\dots,\frac{\sigma+k-1}{k}$.
\end{theorem}
\begin{proof}
By multiplying explicit form of  bivariate Jacobi Konhauser polynomials with $\frac{\Gamma(1+\alpha)\Gamma(1+\beta)}{\Gamma(1+\alpha)\Gamma(1+\beta)}$, we get
\begin{equation}
\mathcal{\prescript{}{\kappa}P}_n^{(\alpha,\beta)}(x,y)=\frac{(1+\alpha)_n}{\Gamma(1+\beta)n!}\sum_{s=0}^n\sum_{r=0}^{n-s}\frac{(-n)_{s+r}(1+\alpha+\beta+n)_s}{s!r!(1+\alpha)_s(\beta+1)_{\kappa r}}\left(\frac{1-x}{2}\right)^sy^{\kappa r},
\end{equation}
where $(1+\beta)_{\kappa r}=\kappa^{\kappa r}\prod_{j=0}^{\kappa-1}\bigg(\frac{\beta+1+j}{\kappa}\bigg)_r$.
\\
Whence we obtain the result.
\end{proof}

\begin{corollary}
From the explicit forms of  bivariate Jacobi Konhauser polynomials \eqref{jac2}, \eqref{jac3} ,\eqref{jac4} the following  representations are also hold respectively,
\begin{equation*}
\mathcal{\prescript{}{\kappa}P}_n^{(\alpha,\beta)}(x,y)=\frac{(-1)^n(1+\beta)_n}{n!}F_{0,1,\kappa}^{1,1,0}\begin{bmatrix} -n:(1+\alpha+\beta+n);-;&\\ & \frac{1+x}{2};\bigg(\frac{y}{\kappa}\bigg)^\kappa \\ - :(1+\beta);\bigtriangleup(\kappa;\beta+1); & \end{bmatrix},
\end{equation*}

\begin{equation*}
\mathcal{\prescript{}{\kappa}P}_n^{(\alpha,\beta)}(x,y)=\frac{(1+\alpha)_n}{\Gamma(1+\beta)n!}\bigg(\frac{x+1}{2}\bigg)^nF_{0,1,\kappa}^{1,1,0}\begin{bmatrix} -n:(-\beta-n);-;&\\ & \frac{x-1}{x+1};\bigg(\frac{y}{\kappa}\bigg)^\kappa \\ - :(1+\alpha);\bigtriangleup(\kappa;\beta+1); & \end{bmatrix},
\end{equation*}

\begin{equation*}
\mathcal{\prescript{}{\kappa}P}_n^{(\alpha,\beta)}(x,y)=\frac{(1+\beta)_n}{n!}\bigg(\frac{x-1}{2}\bigg)^nF_{0,1,\kappa}^{1,1,0}\begin{bmatrix} -n:(-\alpha-n);-;&\\ & \frac{x+1}{x-1};\bigg(\frac{y}{\kappa}\bigg)^\kappa \\ - :(1+\beta);\bigtriangleup(\kappa;\beta+1); & \end{bmatrix}.
\end{equation*}
\end{corollary}

\begin{definition}
The  functions $E_{\alpha,\beta,\kappa}^{(\gamma_1;\gamma_2)}(x,y)$, which are called bivariate Jacobi Konhauser  Mittag Leffler,  defined by the formula,
\begin{equation}
E_{\alpha,\beta,\kappa}^{(\gamma_1;\gamma_2)}(x,y)=\sum_{s=0}^\infty\sum_{r=0}^\infty\frac{(\gamma_1)_{r+s}(\gamma_2)_{s}x^sy^{\kappa r}}{r!s!\Gamma(\alpha+s)\Gamma(\beta+\kappa r)},    \label{mit}
\end{equation}
$(\alpha, \beta, \gamma_1,\gamma_2 \in\mathbb{C}, \mathcal{R}(\alpha), \mathcal{R}(\beta), \mathcal{R}(\kappa)>0).$
\end{definition}

\begin{remark}
By applying the ratio test for the double series, \eqref{mit} converges absolutely for $\mathcal{R}(\kappa)>0$.
\end{remark}

\begin{corollary}
By comparing equations \eqref{jac} and \eqref{mit}, it can be readily observed that:
\begin{equation}
\mathcal{\prescript{}{\kappa}P}_n^{(\alpha,\beta)}(x,y)=\frac{\Gamma(1+\alpha+n)}{n!}E_{\alpha+1,\beta+1,\kappa}^{(-n;1+\alpha+\beta+n)}(\frac{1-x}{2},y). \label{rel}
\end{equation}
\end{corollary}

\section{Fractional Calculus Approach}

In the present section for the bivariate JKML function $E_{\alpha,\beta,\kappa}^{(\gamma_1;\gamma_2)}(x,y)$ and for the polynomial 
$\mathcal{\prescript{}{\kappa}P}_n^{(\alpha,\beta)}(x,y)$, we compute their images under the action of the R-L double fractional integral and derivative.

First let us recall the definitons of the R-L double fractional integral and derivative respectively,
\begin{equation}
(\prescript{}{z}I_{b^+}^{\kappa}\prescript{}{t}I_{a^+}^{\mu})g(t,z)=\frac{1}{\Gamma(\kappa)\Gamma(\mu)}\int_b^z\int_a^t (t-\epsilon)^{\mu-1}(z-\tau)^{\kappa-1}g(\epsilon,\tau)d\epsilon  d\tau,  \label{int}
\end{equation}
$(t>a ,z>b, \mathcal{R}(\kappa)>0, \mathcal{R}(\mu)>0);$

\begin{equation*}
(\prescript{}{t}D_{a^+}^{\zeta}\prescript{}{z}D_{b^+}^{\tau})g(t,z)=\frac{\partial^n}{\partial x^n}\frac{\partial^m}{\partial y^m}(\prescript{}{t}I_{a^+}^{n-\zeta}\prescript{}{z}I_{b^+}^{m-\tau})g(t,z),
\end{equation*}
 $(t>a, z>b, m=[\mathcal{R}(\tau)]+1, n=[\mathcal{R}(\zeta)]+1).$

\begin{theorem} \label{thm3.1}
For the  bivariate JKML function we have
\begin{equation*}
\begin{aligned}
\bigg(\prescript{}{y}I_{b^+}^{\zeta} {}& \prescript{}{x}I_{a^+}^{\mu}\bigg)\bigg[(x-a)^\alpha (y-b)^\beta E_{\alpha+1,\beta+1,\kappa}^{(\gamma_1;\gamma_2)}(\omega_1(x-a),\omega_2(y-b))\bigg] \\ &
=(x-a)^{\alpha+\mu} (y-b)^{\beta+\zeta}E_{1+\alpha+\mu,1+\beta+\zeta,\kappa}^{(\gamma_1;\gamma_2)}(\omega_1(x-a),\omega_2(y-b)).
\end{aligned}
\end{equation*}
\end{theorem}
\begin{proof}
By replacing $E_{\alpha+1,\beta+1,\kappa}^{(\gamma_1;\gamma_2)}(\omega_1(x-a),\omega_2(y-b))$ with its explicit form in the left hand side and intechanging the order of series and  integral  we get,
\begin{equation*}
\begin{aligned}
\bigg(\prescript{}{y}I_{b^+}^{\zeta} {}& \prescript{}{x}I_{a^+}^{\mu}\bigg)\bigg[(x-a)^\alpha (y-b)^\beta E_{\alpha+1,\beta+1,\kappa}^{(\gamma_1;\gamma_2)}(\omega_1(x-a),\omega_2(y-b))\bigg] \\&
=\frac{1}{\Gamma(\zeta)\Gamma(\mu)}\sum_{s=0}^\infty \sum_{r=0}^\infty\frac{(\gamma_1)_{s+r}(\gamma_2)_{s}\omega_1^s \omega_2^{\kappa r}}{s!r!\Gamma(\alpha+s+1)\Gamma(\beta+\kappa r+1)}\\&
\times \int_a^x (x-t)^{\mu-1}(t-a)^{\alpha+s}dt \int_b^y (y-\tau)^{\zeta-1}(\tau-b)^{\beta+\kappa r}d\tau \\&
=(x-a)^{\alpha+\mu} (y-b)^{\beta+\zeta}\sum_{s=0}^\infty \sum_{r=0}^\infty \frac{(\gamma_1)_{s+r}(\gamma_2)_{s}[\omega_1(x-a)]^s [\omega_2(y-b)]^{\kappa r}}{s!r!\Gamma(\alpha+\mu+s+1)\Gamma(\beta+\zeta+ \kappa r+1)}
\end{aligned}
\end{equation*}
whence the proof is completed.
\end{proof}

\begin{corollary}
For the  polynomials $\mathcal{\prescript{}{\kappa}P}_n^{(\alpha,\beta)}(x,y)$ we have
\begin{equation*}
\begin{aligned}
\bigg(\prescript{}{y}I_{b^+}^{\zeta} {}&\prescript{}{x}I_{a^+}^{\mu}\bigg)\bigg[(x-a)^\alpha (y-b)^\beta\mathcal{\prescript{}{\kappa}P}_n^{(\alpha,\beta)}(1-2\omega_1(x-a),\omega_2(y-b))\bigg] \\&
=\frac{\Gamma(1+\alpha+n)}{\Gamma(1+\alpha+n+\mu)}(x-a)^{\alpha+\mu} (y-b)^{\beta+\zeta}\\& \times\mathcal{\prescript{}{\kappa}P}_n^{(\alpha+\mu,\beta+\zeta)}(1-2\omega_1(x-a),\omega_2(y-b))
\end{aligned}
\end{equation*}
\end{corollary}
\begin{proof}
Proof follows from  \eqref{rel} by using Theorem \ref{thm3.1}.
\end{proof}

\begin{theorem} \label{thm3.2}
For the  Bivariate JKML function we have
\begin{equation*}
\begin{aligned}
\bigg(\prescript{}{y}D_{b^+}^{\zeta} {}& \prescript{}{x}D_{a^+}^{\mu}\bigg)\bigg[(x-a)^\alpha (y-b)^\beta E_{\alpha+1,\beta+1,\kappa}^{(\gamma_1;\gamma_2)}(\omega_1(x-a),\omega_2(y-b))\bigg] \\ &
=(x-a)^{\alpha-\mu} (y-b)^{\beta-\zeta}E_{1+\alpha-\mu,1+\beta-\zeta,\kappa}^{(\gamma_1;\gamma_2)}(\omega_1(x-a),\omega_2(y-b))
\end{aligned}
\end{equation*}
\end{theorem}
\begin{proof}
By substituting the explicit form of $E_{\alpha+1,\beta+1,\kappa}^{(\gamma_1;\gamma_2)}(\omega_1(x-a),\omega_2(y-b))$ on the left-hand side and interchanging the order of summation and integration, we obtain
\begin{equation*}
\begin{aligned}
\bigg(\prescript{}{y}D_{b^+}^{\zeta} {}& \prescript{}{x}D_{a^+}^{\mu}\bigg)\bigg[(x-a)^\alpha (y-b)^\beta E_{\alpha+1,\beta+1,\kappa}^{(\gamma_1;\gamma_2)}(\omega_1(x-a),\omega_2(y-b))\bigg] \\ &
=\frac{1}{\Gamma(m-\zeta)\Gamma(n-\mu)}\sum_{s=0}^\infty \sum_{r=0}^\infty\frac{(\gamma_1)_{s+r}(\gamma_2)_{s}\omega_1^s \omega_2^{\kappa r}}{s!r!\Gamma(\alpha+s+1)\Gamma(\beta+\kappa r+1)}  \\&
D_x^n\int_a^x (x-t)^{n-\mu-1}(t-a)^{\alpha+s}dt D_y^m \int_b^y (y-\tau)^{m-\zeta-1}(\tau-b)^{\beta-1+\kappa r)}d\tau \\&
=(x-a)^{\alpha-\mu} (y-b)^{\beta-\zeta}\sum_{s=0}^\infty \sum_{r=0}^\infty\frac{(\gamma_1)_{s+r}(\gamma_2)_{s}[\omega_1(x-a)]^s [\omega_2(y-b)]^{\kappa r}}{s!r!\Gamma(\alpha-\mu+s+1)\Gamma(\beta-\zeta+\kappa r+1)}
\end{aligned}
\end{equation*}
whence the proof is completed.
\end{proof}

\begin{corollary}
For $\mathcal{\prescript{}{\kappa}P}_n^{(\alpha,\beta)}(x,y)$, we have
\begin{equation*}
\begin{aligned}
\bigg(\prescript{}{y}D_{b^+}^{\zeta} {}&\prescript{}{x}D_{a^+}^{\mu}\bigg)\bigg[(x-a)^\alpha (y-b)^\beta\mathcal{\prescript{}{\kappa}P}_n^{(\alpha,\beta)}(1-2\omega_1(x-a),\omega_2(y-b))\bigg] \\&
=\frac{\Gamma(1+\alpha+n)}{\Gamma(1+\alpha+n-\mu)}(x-a)^{\alpha-\mu} (y-b)^{\beta-\zeta}\\&
\times \mathcal{\prescript{}{\kappa}P}_n^{(\alpha-\mu,\beta-\zeta)}(1-2\omega_1(x-a),\omega_2(y-b)).
\end{aligned}
\end{equation*}
\end{corollary}

\begin{proof}
Proof follows from  \eqref{rel} by using Theorem \ref{thm3.2}.
\end{proof}

\section{Laplace transform and operational representations}
We derive the double Laplace transforms and then give the operational representations for  $E_{\alpha,\beta,\kappa}^{(\gamma_1;\gamma_2)}(x,y)$ and  $\mathcal{\prescript{}{\kappa}P}_n^{(\alpha,\beta)}(x,y)$. And also we write bivariate Jacobi K.onhauser polynomials in terms of Bessel functions.

Now, recall the bivariate Laplace transform

\begin{equation*}
\mathbb{L}_2[h(t,z)](p,q)=\int_0^\infty\int_0^\infty e^{-(pt+qz)}h(t,z)dtdz,   \quad  (\mathcal{R}(p), \mathcal{R}(q)>0).
\end{equation*} 

\begin{theorem}
For $\alpha, \beta, \gamma_1, \gamma_2, \kappa \in \mathbb{C}$ , ,  $|\frac{w_2^\kappa}{p_2^\kappa}|<1$ and $|\frac{w_1p_2^\kappa}{p_1(p_2^\kappa-w_2^\kappa)}|<1$ the following double Laplace transformation holds 
\begin{equation*}
\begin{aligned}
\mathbb{L}_2[x^\alpha y^\beta E_{\alpha+1,\beta+1,\kappa}^{(\gamma_1;\gamma_2)}(w_1x,w_2y)](p_1,p_2){}& =p_1^{-(1+\alpha)}p_2^{-(1+\beta)}\bigg(1-\frac{w_2^\kappa}{p_2^\kappa}\bigg)^{-\gamma_1} \\& \times \prescript{}{2}F_0\bigg[\gamma_1,\gamma_2;-,\frac{w_1p_2^\kappa}{p_1(p_2^\kappa-w_2^\kappa)}\bigg]
\end{aligned}
\end{equation*}
$\mathcal{R}(\alpha), \mathcal{R}(\beta),\mathcal{R}(\gamma_1), \mathcal{R}(\gamma_2)>0$,
$ \mathcal{R}(w_1), \mathcal{R}(w_2), \mathcal{R}(p_1), \mathcal{R}(p_2)>0.$
\end{theorem}

\begin{proof}
Because of the assumptions stated in the theorem, we are justified in interchanging the order of the series and the integral. This operation leads to the following result
\begin{equation*}
\begin{aligned}
\mathbb{L}_2{}& [x^\alpha y^\beta E_{\alpha+1,\beta+1,\kappa}^{(\gamma_1;\gamma_2)}(w_1x,w_2y)](p_1,p_2) \\&
=\sum_{s=0}^\infty \sum_{r=0}^\infty\frac{(\gamma_1)_{s+r}(\gamma_2)_{s}\omega_1^s \omega_2^{\kappa r}}{s!r!\Gamma(\alpha+s+1)\Gamma(\beta+\kappa r+1)}\int_0^\infty\int_0^\infty e^{-p_1x-p_2y}x^{\alpha+s}y^{\beta+\kappa r}dxdy \\&
=\sum_{s=0}^\infty \sum_{r=0}^\infty\frac{(\gamma_1)_{s+r}(\gamma_2)_{s}\omega_1^s \omega_2^{\kappa r}}{s!r!\Gamma(\alpha+s+1)\Gamma(\beta+\kappa r+1)p_1^{1+\alpha}p_2^{1+\beta}}\\& 
\times \int_0^\infty\int_0^\infty e^{-u-v}u^{\alpha+s}v^{\beta+\kappa r}dudv \\&
=\frac{1}{p_1^{1+\alpha}p_2^{1+\beta}}\sum_{s=0}^\infty\frac{(\gamma_1)_s(\gamma_2)_s}{s!}\bigg(\frac{w_1}{p_1}\bigg)^s\sum_{r=0}^\infty\frac{(\gamma_1+s)_r}{r!}\bigg(\frac{w_2}{p_2}\bigg)^{\kappa r} \\&
=p_1^{-(1+\alpha)}p_2^{-(1+\beta)}\bigg(1-\frac{w_2^\kappa}{p_2^\kappa}\bigg)^{-\gamma_1}\sum_{s=0}^\infty\frac{(\gamma_1)_s(\gamma_2)_s}{s!}\bigg(\frac{w_1p_2^\kappa}{p_1(p_2^\kappa-w_2^\kappa)}\bigg)^s.
\end{aligned}
\end{equation*}
\end{proof}

\begin{corollary}
For the polynomials  $\mathcal{\prescript{}{\kappa}P}_n^{(\alpha,\beta)}(x,y)$, we have
\begin{equation*}
\begin{aligned}
\mathbb{L}_2 [x^\alpha y^\beta\mathcal{\prescript{}{\kappa}P}_n^{(\alpha,\beta)}(1-2w_1x,w_2y)]{}&(p_1,p_2) =\frac{\Gamma(1+\alpha+n)}{n!p_1^{1+\alpha}p_2^{1+\beta}}\bigg(1-\frac{w_2^\kappa}{p_2^\kappa}\bigg)^{n} \\&
\times \prescript{}{2}F_0\bigg[-n,1+\alpha+\beta+n;-,\frac{w_1p_2^\kappa}{p_1(p_2^\kappa-w_2^\kappa)}\bigg].
\end{aligned}
\end{equation*}
\end{corollary}

\begin{theorem}
For the bivariate JKML functions $E_{\alpha,\beta,\kappa}^{(\gamma_1;\gamma_2)}(x,y)$ the following operational representation holds
\begin{equation*}
E_{\alpha,\beta,\kappa}^{(\gamma_1;\gamma_2)}(x,y)=x^{1-\alpha}y^{1-\beta}\sum_{s=0}^\infty \sum_{r=0}^\infty \frac{(\gamma_1)_{s+r}(\gamma_2)_{s}}{s!r!}D_x^{-s}D_y^{-\kappa r}\Biggl\{\frac{x^{\alpha-1}y^{\beta-1}}{\Gamma(\alpha)\Gamma(\beta)}\Bigg\},
\end{equation*}
where $D_x^{-1}$ is the inverse operator of the derivative.
\end{theorem}

\begin{proof}
The proof follows directly  from the explicit form of the bivariate JKML functions.
\end{proof}

\begin{corollary}
For the bivariate JKML functions $E_{\alpha,\beta,\kappa}^{(\gamma_1;\gamma_2)}(x,y)$ the following operational representation also holds
\begin{equation*}
\begin{aligned}
E_{\alpha,\beta,\kappa}^{(\gamma_1;\gamma_2)}(x,y){}&=x^{1-\alpha}y^{1-\beta}\bigg[\frac{1}{1-D_y^{-\kappa }}\bigg]^{-\gamma_1}\prescript{}{2}F_0\bigg[\gamma_1,\gamma_2;-;\frac{1}{D_x(1-D_y^{-\kappa })}\bigg]\\&
\times\Biggl\{\frac{x^{\alpha-1}y^{\beta-1}}{\Gamma(\alpha)\Gamma(\beta)}\Bigg\}.
\end{aligned}
\end{equation*}
\end{corollary}

\begin{theorem}
The following operational representation hold true for the polynomials  $\mathcal{\prescript{}{\kappa}P}_n^{(\alpha,\beta)}(x,y)$
\begin{equation*}
\begin{aligned}
\mathcal{\prescript{}{\kappa}P}_n^{(\alpha,\beta)}(x,y){}&=\frac{\Gamma(1+\alpha+n)}{n!}\bigg(\frac{1-x}{2}\bigg)^{-\alpha}y^{-\beta}(1-D_y^{-\kappa})^n \\&
\times\prescript{}{2}F_0\bigg[-n,1+\alpha+\beta+n;-;\frac{-2}{D_x(1-D_y^{-\kappa })}\bigg]\\&
\times \bigg\{\frac{(\frac{1-x}{2})^\alpha y^\beta}{\Gamma(\alpha+1)\Gamma(\beta+1)}\bigg\} \\&
=\frac{\Gamma(1+\alpha+n)}{n!}\bigg(\frac{1-x}{2}\bigg)^{-\alpha}(1-D_y^{-\kappa})^n\\& 
\times y_n\bigg(\frac{1}{2D_x(1-D_y^{-\kappa})};2+\alpha+\beta,1\bigg)  \bigg\{\frac{(\frac{1-x}{2})^\alpha y^\beta}{\Gamma(\alpha+1)\Gamma(\beta+1)}\bigg\},
\end {aligned}
\end{equation*}
where $y_n(x;a,b)=\prescript{}{2}F_0[-n,a-1+n;-,\frac{-x}{b}]$ are the Bessel polynomials.
\end{theorem}

\section{Various properties}

In the present section, we first derive  a generating function for the bivariate Jacobi Konhauser polynomials by means of Kampe de Ferit' s double hypergeometric functions. We then derive Schalfi' s contour integral representations of $E_{\alpha,\beta,\kappa}^{(\gamma_1;\gamma_2)}(x,y)$ and $\mathcal{\prescript{}{\kappa}P}_n^{(\alpha,\beta)}(x,y)$. Additionally, we obtain integral representations for the product of  $E_{\alpha,\beta,\kappa}^{(\gamma_1;\gamma_2)}(x,y)$, which includes a special case for $\mathcal{\prescript{}{\kappa}P}_n^{(\alpha,\beta)}(x,y)$.\\
From the view of the definition of the double hypergeometric series  \cite{dob}
\begin{equation*}
\begin{aligned}
{}&S^{A:B;B'}_{C:D;D'}\begin{bmatrix} [(a):\theta,\phi]:[(b):\psi]:[(b'):\psi'];&\\ & x ;y \\ [(c):\delta,\epsilon]:[(d):\xi]:[(d'):\xi']; & \end{bmatrix} \\&
=\sum_{s=0}^\infty\sum_{r=0}^\infty\frac{\prod_{j=0}^A\Gamma(a_j+s\theta+r\phi)\prod_{j=0}^B\Gamma(b_j+s\psi)\prod_{j=0}^{B'}\Gamma(b_j'+r\psi')}{\prod_{j=0}^C\Gamma(c_j+s\delta+r\epsilon)\prod_{j=0}^D\Gamma(d_j+s\xi)\prod_{j=0}^{D'}\Gamma(b_j'+r\xi')}\frac{x^sy^r}{s!r!}.
\end{aligned}
\end{equation*}
we derive the following result.
\begin{theorem}
The  polynomials $\mathcal{\prescript{}{\kappa}P}_n^{(\alpha,\beta)}(x,y)$  have the following generating function
\begin{equation}
\begin{aligned}
\sum_{n=0}^\infty \frac{(1+\alpha+\beta)_n\Gamma(\beta+1)}{(1+\alpha)_n}\prescript{}{\kappa}{}& P_n^{(\alpha,\beta)}(x,y)t^n = \frac{1}{(1-t)^{1+\alpha+\beta}}\frac{\Gamma(1+\alpha)\Gamma(1+\beta)}{\Gamma(1+\alpha+\beta)}\\&
\times S_{0,1,1}^{1,0,0}\begin{bmatrix} [1+\alpha+\beta:2,1]:-:-;&\\ & \frac{t(x-1)}{2(1-t)^2} ;\frac{ty^{\kappa}}{(t-1)\kappa^{\kappa}} \\ -:\alpha+1:1]:[\beta+1:\kappa]; & \end{bmatrix}.
\end{aligned}
\end{equation}
\end{theorem}

\begin{proof}
 By replacing $\mathcal{\prescript{}{\kappa}P}_n^{(\alpha,\beta)}(x,y)$ with \eqref{jac} in the left henad side and using the properties of series we get,
\begin{equation*}
\begin{aligned}
{}&\sum_{n=0}^\infty \frac{(1+\alpha+\beta)_n\Gamma(\beta+1)}{(1+\alpha)_n}\mathcal{\prescript{}{\kappa}P}_n^{(\alpha,\beta)}(x,y)t^n \\&=\sum_{n=0}^\infty\sum_{s=0}^n\sum_{r=0}^{n-s}\frac{(-n+s)_r(1+\alpha+\beta)_{n+s}}{(n-s)!s!r!(1+\alpha)_s(1+\beta)_{\kappa r}} \bigg(\frac{x-1}{2}\bigg)^2y^{\kappa r} t^n \\&
=\sum_{n=0}^\infty\sum_{s=0}^\infty\sum_{r=0}^\infty\frac{(1+\alpha+\beta)_{n+2s+r}}{s!r!(1+\alpha)_s(1+\beta)_{\kappa r}} \bigg(\frac{tx-t}{2}\bigg)^2(-ty)^{\kappa r} t^n.
\end{aligned}
\end{equation*}
Since $(1+\alpha+\beta)_{n+2s+r}=(1+\alpha+\beta)_{2s+r}(1+\alpha+\beta+2s+r)_{n}$ and $\sum_{n=0}^\infty\frac{(1+\alpha+\beta+2s+r)_{n}t^n}{n!}=(1-t)^{-(1+\alpha+\beta+2s+r)}$ we obtain
\begin{equation*}
\begin{aligned}
{}&\sum_{n=0}^\infty \frac{(1+\alpha+\beta)_n\Gamma(\beta+1)}{(1+\alpha)_n}\mathcal{\prescript{}{\kappa}P}_n^{(\alpha,\beta)}(x,y)t^n \\&= \frac{1}{(1-t)^{1+\alpha+\beta}} \sum_{s=0}^\infty\sum_{r=0}^\infty\frac{(1+\alpha+\beta)_{2s+r}}{s!r!(1+\alpha)_s(1+\beta)_{\kappa r}}\bigg(\frac{t(x-1)}{2(1-t)^2}\bigg)^s\bigg(\frac{ty^{\kappa}}{(1-t)\kappa^{\kappa}}\bigg)^r.
\end{aligned}
\end{equation*}
Consequently, we derive the result.
\end{proof}

Before presenting the integral representations, it is essential to first recall the integral representations of the Gamma function.
\begin{equation}
\Gamma(z)=\int_0^{\infty}e^uu^{z-1}du  \quad (\mathcal{R}(u)>0), \label{gam1}
\end{equation}

\begin{equation}
\frac{1}{\Gamma(z)}=\frac{1}{2\pi i}\int_{-\infty}^{0+}e^uu^{-z}du \quad (arg|u|) \leq  \pi), \label{gam2}
\end{equation}
 where \eqref{gam2} referred to as Hankel's representation \cite{last}. 
\begin{theorem}
The bivariate JKML functions have the following integral representation
\begin{equation*}
E_{\alpha,\beta,\kappa}^{(\gamma_1;\gamma_2)}(x,y)=\frac{-1}{4\pi^2\Gamma(\gamma_2)}\int_0^{\infty}\int_{-\infty}^{0+}\int_{-\infty}^{0+}t^{\gamma_2-1}u^{-\alpha}w^{-\beta}e^{u+w-t}\bigg[1-\bigg(\frac{y}{w}\bigg)^\kappa-\frac{tx}{u}\bigg]^{-\gamma_1}dwdudt.
\end{equation*}
$(\alpha, \beta,\kappa, \gamma_1, \gamma_2 \in \mathbb{C}, \mathcal{R}(\alpha), \mathcal{R}(\beta), \mathcal{R}(\kappa)>0, |\frac{y^\kappa}{w^\kappa}|<1$ and $|\frac{txw^\kappa}{u(w^\kappa-y^\kappa)}|<1$)
\end{theorem}

\begin{proof}
Using formulas \eqref{gam1}, \eqref{gam2} and the relation $(a)_{s+r}=(a)_s(a+s)_r$ we obtain that 
\begin{equation*}
\begin{aligned}
E_{\alpha,\beta,\kappa}^{(\gamma_1;\gamma_2)}(x,y){}&=\frac{-1}{4\pi^2\Gamma(\gamma_2)}\int_0^{\infty}\int_{-\infty}^{0+}\int_{-\infty}^{0+}t^{\gamma_2-1}u^{-\alpha}w^{-\beta}e^{u+w-t}\\& 
\times \sum_{s=0}^\infty\frac{(\gamma_1)_s}{s!}\bigg(\frac{tx}{u}\bigg)^s\sum_{r=0}^\infty\frac{(\gamma_1+s)_r}{r!}\bigg(\frac{y}{w}\bigg)^{\kappa r}dwdudt \\&
=\frac{-1}{4\pi^2\Gamma(\gamma_2)}\int_0^{\infty}\int_{-\infty}^{0+}\int_{-\infty}^{0+}t^{\gamma_2-1}u^{-\alpha}w^{-\beta}e^{u+w-t}\bigg[1-\bigg(\frac{y}{w}\bigg)^\kappa\bigg]^{-\gamma_1}\\& 
\times \sum_{s=0}^\infty\frac{(\gamma_1)_s}{s!}\bigg(\frac{txw^\kappa}{u(w^\kappa-y^\kappa)}\bigg)^s.
\end{aligned}
\end{equation*}
\end{proof}

\begin{corollary}
The polynomials  $\mathcal{\prescript{}{\kappa}P}_n^{(\alpha,\beta)}(x,y)$  have the following integral representation 
\begin{equation*}
\begin{aligned}
\mathcal{\prescript{}{\kappa}P}_n^{(\alpha,\beta)}(x,y){}& =\frac{-\Gamma(1+\alpha+n)}{4\pi^2\Gamma(1+\alpha+\beta+n)n!}\int_0^{\infty}\int_{-\infty}^{0+}\int_{-\infty}^{0+}t^{\alpha+\beta+n}u^{-(1+\alpha)}w^{(1+\beta)}e^{u+w-t} \\&
\times \bigg[1-\bigg(\frac{y}{w}\bigg)^\kappa-\frac{t}{u}\bigg(\frac{1-x}{2}\bigg)\bigg]^ndwdudt.
\end{aligned}
\end{equation*}
\end{corollary}

We  utilize the following integral representation in order to prove the subsequent theorem.
\begin{equation}
B(\epsilon, \xi)=\frac{\Gamma(\epsilon)\Gamma(\xi)}{\Gamma(\epsilon+\xi)}=\int_0^1u^{\epsilon-1}(1-u)^{\xi-1}    \label{beta}
\end{equation}
$(min\{\mathcal{R}(\epsilon),\mathcal{R}(\xi)\}>0)$
\begin{theorem}
For the bivariate JKML functions we have
\begin{equation*}
\begin{aligned}
{} &E_{\alpha_1,\beta_1,\kappa}^{(\gamma_1;\gamma_2)}(x,y)E_{\alpha_2,\beta_2,\kappa}^{(\sigma_1;\sigma_2)}(x,y)\\&
 =\frac{1}{16\pi^4\Gamma(\gamma_1)\Gamma(\gamma_2)\Gamma(\sigma_1)\Gamma(\sigma_2)}\int_0^1t^{\gamma_1}(1-t)^{\sigma_1-1}\int_0^1u^{\gamma_2}(1-u)^{\sigma_2-1} \\&
\times\bigg(\int_{-\infty}^{0+}\int_{-\infty}^{0+}\int_{-\infty}^{0+}\int_{-\infty}^{0+}e^{\epsilon+\delta+w+\xi}\epsilon^{-\alpha_1}\delta^{-\beta_1}\xi^{-\alpha_2}w^{-\beta_2} \\&
\times S_{0,0,0}^{1,1,0}\begin{bmatrix} [\gamma_1+\sigma_1:1:1]:[\gamma_2+\sigma_2:1:0]:-;&\\ & x[(1-t)(1-u)\xi^{-1}+tu\epsilon^{-1}] ;y^{\kappa}[(1-t)w^{-\kappa}+t\delta^{-\kappa} \\ - :-:-; & \end{bmatrix} \\&
d\epsilon d\sigma dw d\xi\bigg)dudt,
\end{aligned}
\end{equation*}
$(|arg(\epsilon)|,|arg(\sigma)|,|arg(w)|,|arg(\xi)|\leq \pi;  min\{\mathcal{R}(\alpha_1),\mathcal{R}(\alpha_2)\}>0; \\ min\{\mathcal{R}(\kappa),\mathcal{R}(\gamma_1),\mathcal{R}(\gamma_2),\mathcal{R}(\sigma_1)\mathcal{R}(\sigma_2)\}>0).$ \\
\end{theorem}

\begin{proof}
From the definition \eqref{mit}, we obtain 
\begin{equation*}
\begin{aligned}
{} &E_{\alpha_1,\beta_1,\kappa}^{(\gamma_1;\gamma_2)}(x,y)E_{\alpha_2,\beta_2,\kappa}^{(\sigma_1;\sigma_2)}(x,y) \\& =\sum_{s=0}^\infty\sum_{r=0}^\infty\frac{(\gamma_1)_{r+s}(\gamma_2)_{s}x^ry^{\kappa s}}{r!s!\Gamma(\alpha_1+r)\Gamma(\beta_1+\kappa s)}\sum_{p=0}^\infty\sum_{q=0}^\infty\frac{(\sigma_1)_{q+p}(\sigma_2)_{p}x^qy^{\kappa p}}{p!q!\Gamma(\alpha_2+q)\Gamma(\beta_2+\kappa p)} \\&
=\frac{1}{\Gamma(\gamma_1)\Gamma(\gamma_2)\Gamma(\sigma_1)\Gamma(\sigma_2)}\sum_{p=0}^\infty\sum_{q=0}^\infty\frac{\Gamma(\gamma_1+\sigma_1+p+q)\Gamma(\gamma_2+\sigma_2+q)x^qy^{\kappa p}}{p!q!} \\&
\times \sum_{r=0}^q\sum_{s=0}^p\binom{q}{r}\binom{p}{s}\frac{B(\gamma_1+s+r,\sigma_1+p+q-s-r)B(\gamma_2+r,\sigma_2+q-r)}{\Gamma(\alpha_1+r)\Gamma(\beta_1+\kappa s)\Gamma(\alpha_2+q-r)\Gamma(\beta_2+\kappa(p-s))}
\end{aligned}
\end{equation*}
Now, by applying the formulas \eqref{gam2} and \eqref{beta} we derive the result of the theorem.
\end{proof}

\begin{corollary}
For the polynomials $\mathcal{\prescript{}{\kappa}P}_n^{(\alpha_1,\beta_1)}(x,y)$ we have
\begin{equation*}
\begin{aligned}
{}& \mathcal{\prescript{}{\kappa}P}_n^{(\alpha_1,\beta_1)}(x,y)\mathcal{\prescript{}{\kappa}P}_m^{(\alpha_2,\beta_2)}(x,y)\\& 
=\frac{\Gamma(1+\alpha_1+n)\Gamma(1+\alpha_2+m)}{16\pi^4\Gamma(1+\alpha_1+\beta_1+n)\Gamma(1+\alpha_2+\beta_2+m)}\int_0^\infty\int_0^\infty e^{-t-u}t^{\alpha_1+\beta_1+n}u^{\alpha_2+\beta_2+m} \\&
\times\bigg(\int_{-\infty}^{0+}\int_{-\infty}^{0+}\int_{-\infty}^{0+}\int_{-\infty}^{0+}e^{\epsilon+\delta+w+\xi}\epsilon^{-(1+\alpha_1)}\delta^{-(1+\beta_1)}\xi^{-(1+\alpha_2)}w^{-(1+\beta_2)}\\ &
\times \bigg[\frac{\delta^\kappa-y^\kappa}{\delta^\kappa}-\frac{t(1-x)}{2\epsilon}\bigg]^n\bigg[\frac{w^\kappa-y^\kappa}{w^\kappa}-\frac{u(1-x)}{2\xi}\bigg]^md\epsilon d\sigma dw d\xi\bigg)dudt.
\end{aligned}
\end{equation*}
$(|arg(\epsilon)|,|arg(\sigma)|,|arg(w)|,|arg(\xi)|\leq \pi;  min\{\mathcal{R}(\alpha_1+1),\mathcal{R}(\alpha_2+1)\}>0; \\ min\{\mathcal{R}(\kappa)\}>0; min\{\mathcal{R}(-n),\mathcal{R}(-m)\}>0)$ 
\end{corollary}

\section{Integral operator with $E_{\alpha,\beta,\kappa}^{(\gamma_1;\gamma_2)}(x,y)$}

This section is related with the following operator
\begin{equation}
\begin{aligned}
\bigg(\xi_{\alpha,\beta,\kappa;w_1,w_2,b^+,d^+}^{(\gamma_1;\gamma_2)}f\bigg)(x,y){}&=\int_d^y\int_b^x(x-t)^{\alpha-1}(y-u)^{\beta-1}\\& \times E_{\alpha,\beta,\kappa}^{(\gamma_1;\gamma_2)}(w_1(x-t),w_2(y-u))f(t,u)dtdu,\\& \label{operator}
(x>b,y>d)
\end{aligned}
\end{equation}
which is a fractional double integral operator and when $\gamma_1=\gamma_2=0$, it reduces to the  operator defined in \eqref{int},

\begin{equation*}
\bigg(\xi_{\alpha,\beta,\kappa;w_1,w_2,b^+,d^+}^{(0;0)}f\bigg)(x,y)=\bigg(\prescript{}{x}I_{b^+}^{\alpha} \prescript{}{y}I_{d^+}^{\beta}f\bigg)(x,y).
\end{equation*}

$L((b,a)\times(d,c))$ is the space of absolutely integrable functions which is defined as follows
\begin{equation*}
L((b,a)\times(d,c))=\bigg\{g:||g||_1=\int_b^a\int_d^c|g(x,y)|dydx<\infty\bigg\}.
\end{equation*}

\begin{theorem}
The double in tegral operator $\xi_{\alpha,\beta,\kappa;w_1,w_2,b^+,d^+}^{(\gamma_1;\gamma_2)}$ is bounded in the space $L((b,a)\times(d,c))$,
\begin{equation}
\bigg|\bigg|\xi_{\alpha,\beta,\kappa;w_1,w_2,b^+,d^+}^{(\gamma_1;\gamma_2)}f\bigg|\bigg|_1\leq K||f||_1,
\end{equation}
where the constant K is independent of $f$ given by
\begin{equation*}
K=(a-b)^{\mathcal{R}(\alpha)}(c-d)^{\mathcal{R}(\beta)}\sum_{s=0}^\infty\sum_{r=0}^\infty\frac{|(\gamma_1)_{r+s}||(\gamma_2)_{s}||w_1(a-b)|^s|w_2(c-d)|^{\kappa r}}{r!s!(\mathcal{R}(\alpha)+s)(\mathcal{R}(\beta)+\kappa r)|\Gamma(\alpha+s)||\Gamma(\beta+\kappa r)|}.
\end{equation*}
\end{theorem}

\begin{proof}
By applying Fubini's Theorem, we obtain
\begin{equation*}
\begin{aligned}
{}& \bigg|\bigg|\xi_{\alpha,\beta,\kappa;w_1,w_2,b^+,d^+}^{(\gamma_1;\gamma_2)}f\bigg|\bigg|_1 \\&
\leq\int_b^a\int_d^c|f(t,u)| \\&
\times \bigg(\int_t^a\int_u^c(x-t)^{\mathcal{R}(\alpha)-1}(y-u)^{\mathcal{R}(\beta)-1}\bigg|E_{\alpha,\beta,\kappa}^{(\gamma_1;\gamma_2)}(w_1(x-t),w_2(y-u))\bigg|dydx\bigg)dudt \\&
=\int_b^a\int_d^c|f(t,u)|\bigg((\int_0^{a-t}\int_0^{c-u}\nu^{Re(\alpha)-1}\tau^{Re(\beta)-1}\bigg|E_{\alpha,\beta,\kappa}^{(\gamma_1;\gamma_2)}(w_1\nu,w_2\tau)\bigg|d\tau d\nu\bigg)dudt \\&
\leq \int_b^a\int_d^c|f(t,u)|\bigg((\int_0^{a-b}\int_0^{c-d}\nu^{Re(\alpha)-1}\tau^{Re(\beta)-1}\bigg|E_{\alpha,\beta,\kappa}^{(\gamma_1;\gamma_2)}(w_1\nu,w_2\tau)\bigg|d\tau d\nu\bigg)dudt \\&
\leq \sum_{s=0}^\infty\sum_{r=0}^\infty\frac{|(\gamma_1)_{r+s}||(\gamma_2)_{s}||w_1|^s|w_2|^{\kappa r}}{r!s!|\Gamma(\alpha+s)||\Gamma(\beta+\kappa r)|}\\&
\times \int_0^{a-b}\int_0^{c-d}\nu^{\mathcal{R}(\alpha)+s-1}\tau^{\mathcal{R}(\beta)+\kappa r-1}d\tau d\nu||f||_1 \\&
=K||f||_1.
\end{aligned}
\end{equation*}
\end{proof}

\begin{remark}
Since the series \eqref{mit} is absolutely convergent for all $x$ and $y$ for $\mathcal{R}(\kappa)>0$, the constant $K$ is finite.
\end{remark}

The next result of this section is a theorem that establishes a fundamental relationship between the integral operator, defined by equation \eqref{operator}  and the R-L fractional calculus, expressed in the form of an infinite series.

\begin{theorem}
For any $\alpha, \beta, \gamma_1, \gamma_2, \kappa, w_1, w_2 \in \mathbb{C}$, $\mathcal{R}(\alpha), \mathcal{R}(\beta), \mathcal{R}(\kappa)>0$ and $f(t,z) \in L((b,a)\times(d,c))$, the representation of the operator given by equation \eqref{operator} is as follows:
\begin{equation}
\xi_{\alpha,\beta,\kappa;w_1,w_2,b^+,d^+}^{(\gamma_1;\gamma_2)}f(t,z)=\sum_{s=0}^\infty\sum_{r=0}^\infty\frac{(\gamma_1)_{r+s}(\gamma_2)_{s}w_1^sw_2^{\kappa r}}{r!s!}\prescript{}{z}I_{d^+}^{\beta+\kappa r}\prescript{}{t}I_{b^+}^{\alpha+s}f(t,z). \label{series}
\end{equation}
\end{theorem}

\begin{proof}
From the local uniform convergence of the series \eqref{mit} the  interchange of summation and integration in formula \eqref{operator} gives
\begin{equation*}
\begin{aligned}
\bigg({}&\xi_{\alpha,\beta,\kappa;w_1,w_2,b^+,d^+}^{(\gamma_1;\gamma_2)}f\bigg)(x,y) \\&
=\int_d^y\int_b^x(x-t)^{\alpha-1}(y-u)^{\beta-1}E_{\alpha,\beta,\kappa}^{(\gamma_1;\gamma_2)}(w_1(x-t),w_2(y-u))\\& \times f(t,u)dtdu \\&
=\int_d^y\int_b^x(x-t)^{\alpha-1}(y-u)^{\beta-1} \sum_{s=0}^\infty\sum_{r=0}^\infty\frac{(\gamma_1)_{r+s}(\gamma_2)_{s}[w_1(x-t)]^s[w_2(y-u)]^{\kappa r}}{r!s!\Gamma(\alpha+s)\Gamma(\beta+\kappa r)}\\& \times f(t,u)dtdu \\&
= \sum_{s=0}^\infty\sum_{r=0}^\infty\frac{(\gamma_1)_{r+s}(\gamma_2)_{s}w_1^sw_2^{\kappa r}}{r!s!}\frac{1}{\Gamma(\alpha+s)\Gamma(\beta+\kappa r)}  \\& \times \int_d^y\int_b^x(x-t)^{\alpha+s-1}(y-u)^{\beta+\kappa r-1} f(t,u)dtdu \\&
=\sum_{s=0}^\infty\sum_{r=0}^\infty\frac{(\gamma_1)_{r+s}(\gamma_2)_{s}w_1^rw_2^{\kappa s}}{r!s!}\prescript{}{y}I_{d^+}^{\beta+\kappa r}\prescript{}{x}I_{b^+}^{\alpha+s}f(x,y).
\end{aligned}
\end{equation*}
\end{proof}

In the following corollary, we demonstrate the composition of R-L fractional integrals and derivatives with the integral operator \eqref{operator}.

\begin{corollary}
 For all functions $f(x,y) \in L((b,a)\times(d,c))$ the following composition relationships are valid
\begin{equation*}
\begin{aligned}
\big(\prescript{}{y}I_{d^+}^{\zeta}\prescript{}{x}I_{b^+}^{\mu}\xi_{\alpha,\beta,\kappa;w_1,w_2,b^+,d^+}^{(\gamma_1;\gamma_2)}f\big)(x,y){}& =\big(\xi_{\alpha+\mu,\beta+\zeta,\kappa;w_1,w_2,b^+,d^+}^{(\gamma_1;\gamma_2)}f\big)(x,y) \\&
=\big(\xi_{\alpha,\beta,\kappa;w_1,w_2,b^+,d^+}^{(\gamma_1;\gamma_2)}\prescript{}{y}I_{d^+}^{\zeta}\prescript{}{x}I_{b^+}^{\mu}f\big)(x,y)
\end{aligned}
\end{equation*}

\begin{equation*}
\begin{aligned}
(\prescript{}{y}D_{d^+}^{\zeta}\prescript{}{x}D_{b^+}^{\mu}\xi_{\alpha,\beta,\kappa;w_1,w_2,b^+,d^+}^{(\gamma_1;\gamma_2)}f)(x,y){}&=(\xi_{\alpha-\mu,\beta-\zeta,\kappa;w_1,w_2,b^+,d^+}^{(\gamma_1;\gamma_2)}f)(x,y) \\&
=(\xi_{\alpha,\beta,\kappa;w_1,w_2,b^+,d^+}^{(\gamma_1;\gamma_2)}\prescript{}{y}D_{d^+}^{\zeta}\prescript{}{x}D_{b^+}^{\mu}f)(x,y)
\end{aligned}
\end{equation*}
\end{corollary}

To demonstrate the efficiency of the proposed series formula \eqref{series}, we provide a simple method for calculating and applying the integral operator \eqref{operator} to both a power function and an exponential function. This alternative method offers a substantial improvement over the use of the original definition \eqref{operator}.

\begin{theorem}
The integral operator $\xi_{\alpha,\beta,\kappa;w_1,w_2,b^+,d^+}^{(\gamma_1;\gamma_2)}$ applied to a power function is expressed as follows:
\begin{equation}
\begin{aligned}
\xi_{\alpha,\beta,\kappa;w_1,w_2,b^+,d^+}^{(\gamma_1;\gamma_2)}[(x-b)^\mu(y-d)^\zeta]{}&=\Gamma(\mu+1)\Gamma(\zeta+1)(x-b)^{\mu+\alpha}(y-d)^{\zeta+\beta}\\&
\times E_{\alpha+\mu+1,\beta+\zeta+1,\kappa}^{(\gamma_1;\gamma_2)}[w_1(x-b),w_2(y-d)]. \label{ex}
\end{aligned}
\end{equation}
\end{theorem}

\begin{proof}
The R-L fractional integral of a power function is given by
\begin{equation*}
\prescript{}{y}I_{d^+}^{\beta}(y-d)^\zeta=\frac{\Gamma(\zeta+1)}{\Gamma(\zeta+\beta+1)}(y-d)^{\zeta+\beta}, \quad \mathcal{R}(\zeta)>-1,  \quad \mathcal{R}(\beta)>0,
\end{equation*}
\begin{equation*}
\prescript{}{x}I_{b^+}^{\alpha}(x-b)^{\mu}= \frac{\Gamma(\mu+1)}{\Gamma(\mu+\alpha+1)}(x-b)^{\alpha+\mu}, \quad \mathcal{R}(\mu)>-1,  \quad \mathcal{R}(\alpha)>0.
\end{equation*}

For $\mathcal{R}(\zeta)>-1$ and $\mathcal{R}(\mu)>-1$ by using above relations with series formula \eqref{series} we obtain
\begin{equation*}
\begin{aligned}
{}&\xi_{\alpha,\beta,\kappa;w_1,w_2,b^+,d^+}^{(\gamma_1;\gamma_2)}[(x-b)^\mu(y-d)^\zeta] \\& =\sum_{s=0}^\infty\sum_{r=0}^\infty\frac{(\gamma_1)_{r+s}(\gamma_2)_{s}w_1^sw_2^{\kappa r}}{r!s!}\prescript{}{y}I_{d^+}^{\beta+\kappa r}\prescript{}{x}I_{b^+}^{\alpha+s}[(x-b)^\mu(y-d)^\zeta] \\&
=\sum_{s=0}^\infty\sum_{r=0}^\infty\frac{(\gamma_1)_{r+s}(\gamma_2)_{s}w_1^sw_2^{\kappa r}}{r!s!}\frac{\Gamma(\zeta+1)\Gamma(\mu+1)}{\Gamma(\zeta+\beta+\kappa r+1)\Gamma(\mu+\alpha+r+1)} \\&
\times (x-b)^{\alpha+\mu+s}(y-d)^{\zeta+\beta+\kappa s} \\&
=\Gamma(\mu+1)\Gamma(\zeta+1)(x-b)^{\mu+\alpha}(y-d)^{\zeta+\beta} \\&
\times \sum_{s=0}^\infty\sum_{r=0}^\infty\frac{(\gamma_1)_{r+s}(\gamma_2)_{s}[w_1(x-b)]^s[w_2(y-d)]^{\kappa r}}{\Gamma(\mu+\alpha+1)\Gamma(\zeta+\beta+1)r!s!}
\end{aligned}
\end{equation*}
Thus, the result in equation \eqref{ex} follows.
\end{proof}

\begin{theorem} \label{exp}
The integral operator $\xi_{\alpha,\beta,\kappa;w_1,w_2,b^+,d^+}^{(\gamma_1;\gamma_2)}$ applied to an exponential function, with left limits $b=-\infty$ and $d=-\infty$  is given by the formula:
\begin{equation*}
\begin{aligned}
\xi_{\alpha,\beta,\kappa;w_1,w_2,-\infty,-\infty}^{(\gamma_1;\gamma_2)}[e^{\delta x+\sigma y}]{}& =\delta^{-\alpha}\sigma^{-\beta}e^{\delta x+\sigma y}\bigg[1-\bigg(\frac{w_2}{\sigma}\bigg)^{\kappa }\bigg]^{-\gamma_1}\\&
\times \prescript{}{2}F_0\bigg[\gamma_1,\gamma_2;-;\bigg(\frac{\sigma^\kappa w_1}{\delta(\sigma^\kappa-w_2^\kappa)}\bigg)\bigg],
\end{aligned}
\end{equation*}
where $|\frac{w_1}{\delta}|<1$.
\end{theorem}

\begin{proof}
The R-L fractional integral of an exponential function is given by
\begin{equation}
{}_{\phantom{1}y}I_{-\infty}^{\beta}e^{\sigma y}=\sigma^{-\beta}e^{\sigma y}, \quad \mathcal{R}(\sigma)>0, \mathcal{R}(\beta)>0. \label{ab}
\end{equation}
Using equation \eqref{ab} together with the series formula \eqref{series}, and assuming $\mathcal{R}(\sigma) > 0$ and $\mathcal{R}(\delta) > 0$, it follows that:
\begin{equation*}
\begin{aligned}
{}&\xi_{\alpha,\beta,\kappa;w_1,w_2,-\infty,-\infty}^{(\gamma_1;\gamma_2)}[e^{\delta x+\sigma y}]\\&
=\sum_{s=0}^\infty\sum_{r=0}^\infty\frac{(\gamma_1)_{r+s}(\gamma_2)_{s}w_1^sw_2^{\kappa r}}{r!s!}\prescript{}{y}I_{d^+}^{\beta+\kappa r}\prescript{}{x}I_{b^+}^{\alpha+s}[e^{\delta x+\sigma y}] \\&
=\delta^{-\alpha}\sigma^{-\beta}e^{\delta x+\sigma y}\sum_{s=0}^\infty\sum_{r=0}^\infty\frac{(\gamma_1)_{r+s}(\gamma_2)_{s}(\frac{w_1}{\delta})^s(\frac{w_2}{\sigma})^{\kappa r}}{r!s!} \\&
=\delta^{-\alpha}\sigma^{-\beta}e^{\delta x+\sigma y}\sum_{s=0}^\infty\frac{(\gamma_1)_{s}(\gamma_2)_{s}(\frac{w_1}{\delta})^s}{s!}\sum_{r=0}^\infty\frac{(\gamma_1+s)_{r}(\frac{w_2}{\sigma})^{\kappa r}}{r!} \\&
=\delta^{-\alpha}\sigma^{-\beta}e^{\delta x+\sigma y}\bigg[1-\bigg(\frac{w_2}{\sigma}\bigg)^{\kappa }\bigg]^{-\gamma_1}\sum_{s=0}^\infty\frac{(\gamma_1)_{s}(\gamma_2)_{s}}{s!}\bigg(\frac{\sigma^\kappa w_1}{\delta(\sigma^\kappa-w_2^\kappa)}\bigg)
^s.
\end{aligned}
\end{equation*}
\end{proof}

Next, we focus on the Laplace transform. Since the  operator \eqref{operator} involves a double integral, it is necessary to consider a two-variable Laplace transform.

\begin{theorem}
Under the conditions $\gamma_1, \gamma_2, \kappa, \alpha, \beta, w_1, w_2 \in\mathbb{C}, \mathcal{R}(\beta)>0, \mathcal{R}(\alpha)>0,  \mathcal{R}(\kappa)>0$ and for any function $\psi \in L(( 0,a) \times (0,c))$, we have
\begin{equation*}
\begin{aligned}
\mathbb{L}_2\bigg(\xi_{\alpha,\beta,\kappa;w_1,w_2,0,0}^{(\gamma_1;\gamma_2)}f\bigg)(p,q){}& =\frac{1}{p^\alpha q^\beta}\bigg[1-\bigg(\frac{w_2}{q}\bigg)^{\kappa }\bigg]^{-\gamma_1}\\& \times \prescript{}{2}F_0\bigg[\gamma_1,\gamma_2;-;\bigg(\frac{q^\kappa w_1}{p(q^\kappa-w_2^\kappa)}\bigg)\bigg]\mathbb{L}_2f(p,q),
\end{aligned}
\end{equation*}
where $p, q \in\mathbb{C}$ such that $\mathcal{R}(q)>0, \mathcal{R}(p)>0$ and $|\frac{w_1}{q}|<1$.
\end{theorem}

\begin{proof}
By using \eqref{series} we obtain,
\begin{equation*}
\begin{aligned}
\mathbb{L}_2\bigg(\xi_{\alpha,\beta,\kappa;w_1,w_2,0,0}^{(\gamma_1;\gamma_2)}f\bigg)(p,q) 
{}&=\mathbb{L}_2\bigg[\sum_{s=0}^\infty\sum_{r=0}^\infty\frac{(\gamma_1)_{r+s}(\gamma_2)_{s}w_1^sw_2^{\kappa r}}{r!s!}\prescript{}{y}I_{0}^{\beta+\kappa r}\prescript{}{x}I_{0}^{\alpha+s}f\bigg](p,q) \\&
=\sum_{s=0}^\infty\sum_{r=0}^\infty\frac{(\gamma_1)_{r+s}(\gamma_2)_{s}w_1^sw_2^{\kappa r}}{r!s!}\mathbb{L}_2\prescript{}{y}I_{0}^{\beta+\kappa r}\prescript{}{x}I_{0}^{\alpha+s}f\bigg](p,q)  \\&
=\frac{1}{p^\alpha q^\beta}\sum_{s=0}^\infty\sum_{r=0}^\infty\frac{(\gamma_1)_{r+s}(\gamma_2)_{s}(\frac{w_1}{p})^s(\frac{w_2}{q})^{\kappa r}}{r!s!}\mathbb{L}_2f(p,q)
\end{aligned}
\end{equation*}
the remaining manipulation of the series follows precisely the same steps as in the proof of above Theorem. We have utilized the standard result \cite{lap} that the double R-L integral has a double Laplace transform given by:
\begin{equation*}
\mathbb{L}_2\bigg[{}_{\phantom{1}y}I_{0^+}^{\beta}{}_{\phantom{1}x}I_{0^+}^{\alpha}f(x,y)\bigg]=\frac{1}{p^\alpha q^\beta}\mathbb{L}_2f(x,y), 
\end{equation*}
$ \mathcal{R}(\alpha)>0, \mathcal{R}(\beta)>0, \mathcal{R}(p)>0, \mathcal{R}(q)>0.$
\end{proof}

\section*{Conclusion}
The novel methodology for constructing two-variable biorthogonal polynomial families proposed by utilizing one-variable biorthogonal and orthogonal polynomial families. The primary aim of this study was to introduce a novel class of two-variable biorthogonal polynomials, namely the bivariate Jacobi Konhauser polynomials. We investigated several fundamental properties of these polynomials.  Furthermore, we explored the images of these polynomials under the Laplace transform, as well as fractional integral and derivative operators. In relation to these polynomials, we defined a new type of bivariate Jacobi-Konhauser-Mittag-Leffler (JKML) function and derived similar properties for the function. Additionally, we introduced an integral operator whose kernel contained the bivariate JKML function. We computed the images of power and exponential functions under the action of that operator. Lastly,we investigated the image of the operator under the Laplace transform.



\begin{thebibliography}{6}

\bibitem{carl}
 Carlitz L. A note on certain biorthogonal polynomials. Pac.J. Math.1968;24:425-430.

\bibitem{kon}
 Konhauser JDE. Biorthogonal polynomials suggested by the Laguerre polynomials. Pasific J. Math.1967;21: 303-314.

\bibitem{jde}
 Konhauser JDE. Some properties of biorthogonal polynomials. J. Math. Anal.Appl. 1965;11: 242-260.

\bibitem{hc}
Madhekar HC, thakare NK. Biorthogonal polynomials suggested by the Jacobi polynomials.  Pasific J. Math. 1982;100:417-424.



\bibitem{ilk}
Özarslan MA, Elidemir İO. Fractional Calculus Operator Emerging from the 2D Biorthogonal Hermite Konhauser Polynomials, 2023 (arXiv:2404.00035). 

\bibitem{1}
 Bin-Saad Maged G. Associated Laguerre-Konhauser polynomials quasi-monomialitY and operational identities. J. Math. Anal.Appl.  2006;324:1438-1448.

\bibitem{cemo}
 Özarslan MA, Kürt C. Bivariate Mittag-Leffler functions arising in the solutions of convolution integral equation with 2D-Laguerre-Konhauser polynomials in the kernel. Appl. Math. Comp. 2019;347:631-644.

\bibitem{book}
Gorenflo R, Kilbas  AA, Mainardi F,  Rogosin S V. Mittag-Leffler functions, related topics and applications. 2nd edition: Berlin: Springer;2020.

\bibitem{luchko}
Kiryakova VS,  Luchko Y F. (2010, November). The Multi‐index Mittag‐Leffler Functions and their Applications for Solving Fractional Order Problems in Applied Analysis. In AIP Conference Proceedings. American Institute of PhYsics. 2010;1301: 597 - 613.

\bibitem{luchko2}
Luchko Y. The four-parameters Wright function of the second kind and its applications in FC. Mathematics. 2020;8: 970.

\bibitem{prab}
Prabhakar TR. A singular integral equation with a generalized Mittag-Leffler function in the kernel. Yokohama math. J. 1971;19:7-15.

\bibitem{kilbas}
Kilbas  AA, Saigo  M,  Saxena R K. Generalized Mittag-Leffler function and generalized fractional calculus operators. Integral transforms and Special Functions. 2004;15:31-49.

\bibitem{natur}
Fernandez A, Kürt C,  Özarslan  MA. A naturally emerging bivariate Mittag-Leffler function and associated fractional-calculus operators. Computational and Applied Mathematics. 2020;39:1-27.

\bibitem{ng}
Ng  YX, Phang  C, Loh   JR, Isah  A. MalaYsia S. Analytical solutions of incommensurate fractional differential equation systems with fractional order $1< \gamma, \beta< 2$ via bivariate Mittag-Leffler functions. behaviour. 2022;13:14.

\bibitem{garg}
Garg  M, Manohar P.  Kalla  S L. (2013). A Mittag-Leffler-type function of two variables. Integral transforms and Special Functions. 2013;24: 934-944.

\bibitem{kilbas2}
Saigo  M,  Kilbas  AA. (1998). On Mittag-Leffler type function and applications. Integral transforms and Special Functions. 1998;7: 97-112.

\bibitem{kilbas3}
Kilbas  AA,  Saigo M. On Mittag-Leffler type function, fractional calculas operators and solutions of integral equations. Integral transforms and Special Functions. 1996;4: 355-370.

\bibitem{shukla}
Shukla  A K,  Prajapati  JC. On a generalization of Mittag-Leffler function and its properties. Journal of mathematical analysis and applications. 2007;336: 797-811.

\bibitem{aa}
 Kilbas AA, Srivastava HM, trujillo JJ. theory and Applications of Fractional Differential Equations. Elsevier: Amsterdam, the Netherlands. 2006.

\bibitem{rk}
 Saxena RK , Kalla SL , Saxena R. Multivariate analogue of generalized Mittag-Leffler function, Int. trans.  Special Funct. 2011;22: 533-548.

\bibitem{ma}
 Özarslan MA. On a singular integral equation including a set of multivariate polynomials suggested by Laguerre polynomials,. pplied Mathematics and Computation. 2014;229: 350-358.

\bibitem{ma1}
Özarslan MA,  Fernandez A. On the fractional calculus of multivariate Mittag-Leffler functions. International Journal of Computer Mathematics. 2022;99:247-273.

\bibitem{ck}
 Kürt C, Fernandez A, Özarslan MA. Two unified families of bivariate Mittag-Leffler functions. Applied Mathematics and Computation. 2023;443:127785.

\bibitem{kc}
Kürt  C, Özarslan MA,  Fernandez A.  On a certain bivariate Mittag‐Leffler function analysed from a fractional‐calculus point of view. Mathematical Methods in the Applied Sciences.  2021;44:2600-2620.

\bibitem{arren}
Özarslan MA,  Fernandez  A. On a five-parameter Mittag-Leffler function and the corresponding bivariate fractional operators. Fractal and Fractional. 2021;5:45.

\bibitem{new}
Isah SS, Fernandez A, Özarslan MA. (2023). On bivariate fractional calculus with general univariate analYtic kernels. Chaos, Solitons  Fractals. 2023;171: 113495.

\bibitem{new2}
Isah  SS, Fernandez A,  Özarslan, MA. On univariate fractional calculus with general bivariate analYtic kernels. Computational and Applied Mathematics, 2023;42(5):228.

\bibitem{new3}
Shahwan  MJ , Bin-Saad  M G,  Al-Hashami A,  Some properties of bivariate Mittag-Leffler function. the Journal of AnalYsis, 2023; 1-21.


\bibitem{kampee}
 Srivastava HM , Daust MC.  A note on the convergence of Kampê de fêriet's double hypergeometric series. Math. Nachr. 1972;53: 151-159.

\bibitem{dob}
Srivastava  HM,  Daoust  MC. On Eulerian integrals associated with Kampé de Fériet's function. Publications de l'Institut Mathématique [Elektronische Ressource], 1969;23:199-202.

\bibitem{last}
Whittaker ET, Watson GN. A course of modern analysis: An introduction to the general theorY of infinite process an of analYtic functions; with an account of the principal transcendental functions ed. Cambridge, London and New York. 1927.



\bibitem{lap}
Samko SG, Kilbas AA,  Marichev OI. Fractional Integrals and Derivatives: Theory and Applications, Taylor and  Francis, London, 2002.

\end{thebibliography}
\end{document}